\documentclass[12pt,reqno]{amsart}
\usepackage{amsmath}
\usepackage{amssymb}
\usepackage[left=3cm,top=3cm,right=3cm,bottom=3cm]{geometry}
\usepackage{epsfig}
\usepackage[normalem]{ulem}
\usepackage[usenames,dvipsnames]{color}
\usepackage{todonotes}

\numberwithin{equation}{section}

\newcommand{\remove}[1]{}

\begin{document}
\newtheorem{theorem}{Theorem}[section]
\newtheorem{lemma}[theorem]{Lemma}
\newtheorem{sublemma}[theorem]{Sub-lemma}
\newtheorem{definition}[theorem]{Definition}
\newtheorem{conjecture}[theorem]{Conjecture}
\newtheorem{proposition}[theorem]{Proposition}
\newtheorem{claim}[theorem]{Claim}
\newtheorem{algorithm}[theorem]{Algorithm}
\newtheorem{corollary}[theorem]{Corollary}
\newtheorem{observation}[theorem]{Observation}
\newtheorem{problem}[theorem]{Open Problem}
\newcommand{\R}{{\mathbb R}}
\newcommand{\N}{{\mathbb N}}
\newcommand{\Z}{{\mathbb Z}}
\newcommand\eps{\varepsilon}
\newcommand{\E}{\mathbb E}
\newcommand{\Prob}{\mathbb{P}}
\newcommand{\pl}{\textrm{C}}
\newcommand{\dang}{\textrm{dang}}
\renewcommand{\labelenumi}{(\roman{enumi})}
\newcommand{\bc}{\bar c}
\newcommand{\cal}[1]{\mathcal{#1}}
\newcommand{\G}{{\cal G}}
\newcommand{\Hc}{{\cal H}}
\newcommand{\Gnd}{\G_{n,d}}
\newcommand{\Gnp}{\G(n,p)}
\renewcommand{\P}{{\cal P}}
\newcommand{\la}{\lambda}
\newcommand{\floor}[1]{\lfloor #1 \rfloor}

\newcommand{\bel}[1]{\be\label{#1}}
\newcommand{\ee}{\end{equation}}
\newcommand{\be}{\begin{equation}}
 \newcommand\eqn[1]{(\ref{#1})}
 \newcommand{\ex}{\E}

\newcommand{\aas}{{a.a.s.}}
\newcommand{\wO}{\widetilde O}
\newcommand{\accessconst}{\gammaconst}
\newcommand{\gammaconst}{9}

\newcommand{\Aconst}{a} 
\newcommand{\Bconst}{b} 
\newcommand{\hatU}{\widehat U}
\newcommand{\Bin}{{\rm Bin}}
\newcommand{\tildeU}{{\widetilde U}}

\newcommand{\scr}{\mathcal}
\newcommand{\mb}{\mathbb}
\newcommand{\til}{\widetilde}

\newcommand{\mil}{\mathit}

\newcommand{\opt}{\textup{opt}}
\newcommand{\lmax}{\textup{lmax}}
\newcommand{\pred}{\textup{pred}}

\newcommand{\Avg}{\text{Avg}}

\newcommand{\schur}{\textup{schur}}

\newcommand{\Enc}{\textup{Enc}}
\newcommand{\Dec}{\textup{Dec}}
\newcommand{\mPi}{\mathit{\Pi}}
\newcommand{\gam}{\gamma}
\newcommand{\Gam}{\Gamma}

\title{Probabilistic Zero Forcing on Random Graphs} 

\author{Sean English}
\address{Department of Mathematics, University of Illinois Urbana-Champaign, Champaign, Il}
\email{\texttt{SEnglish@Illinois.edu}}

\author{Pawe\l{} Pra\l{}at}
\address{Department of Mathematics, Ryerson University, Toronto, ON, Canada}
\email{\texttt{pralat@ryerson.ca}}

\author{Calum MacRury}
\address{Department of Computer Science, University of Toronto, Toronto, ON, Canada}
\email{\tt cmacrury@cs.toronto.edu}


\maketitle

\begin{abstract}
Zero forcing is a deterministic iterative graph coloring process in which vertices are colored either blue or white, and in every round, any blue vertices that have a single white neighbor force these white vertices to become blue. Here we study probabilistic zero forcing, where blue vertices have a non-zero probability of forcing each white neighbor to become blue. 

We explore the propagation time for probabilistic zero forcing on the Erd\H{o}s-R\'eyni random graph $\Gnp$ when we start with a single vertex colored blue. We show that when $p=\log^{-o(1)}n$, then with high probability it takes $(1+o(1))\log_2\log_2n$ rounds for all the vertices in $\Gnp$ to become blue, and when $\log n/n\ll p\leq \log^{-O(1)}n$, then with high probability it takes $\Theta(\log(1/p))$ rounds.
\end{abstract}

\section{Introduction\label{intro}}

Zero forcing is an iterative graph coloring procedure which can model certain real world propagation and search processes such as rumor  spreading. Given a graph $G$ and a set of marked, or blue, vertices $Z\subseteq G$, the process of zero forcing involves the application of the \emph{zero forcing color change rule} in which a blue vertex $u$ forces a non-blue (white) vertex $v$ to become blue if $N(u)\setminus Z=\{v\}$, that is, $u$ forces $v$ to become blue if $v$ is the only white neighbor of $u$.

We say that $Z$ is a \emph{zero forcing set} if when starting with $Z$ as the set of initially blue vertices, after iteratively applying the zero forcing color change rule until no more vertices can force, the entire vertex set of $G$ becomes blue. Note that the order in which forces happen is arbitrary since if $u$ is in a position in which it can force $v$, this property will not be destroyed if other vertices are turned blue. As a result, we may process vertices sequentially (in any order) or all vertices that are ready to turn blue can do so simultaneously. The \emph{zero forcing number}, denoted $z(G)$, is the cardinality of the smallest zero forcing set of $G$.

Zero forcing has sparked a lot of interest recently. Some work has been done on calculating or bounding the zero forcing number for specific structures such as graph products \cite{AIM2008}, graphs with large girth \cite{DK2015} and random graphs \cite{BBEMP2018,KKB2019}, while others have studied variants of zero forcing such as connected zero forcing \cite{BH2017} or positive semi-definite zero forcing \cite{BBFHHSVV2010}. 

In the present paper we will be mainly concerned  with a parameter associated with zero forcing known as the \emph{propagation time}, which is the fewest number of rounds necessary for a zero forcing set of size $z(G)$ to turn the entire graph blue. More formally, given a graph $G$ and a zero forcing set $Z$, we generate a finite sequence of sets  $Z_0\subsetneq Z_1\subsetneq \dots\subsetneq Z_t$, where $Z_0=Z$, $Z_t=V(G)$, and given $Z_i$, we define $Z_{i+1}=Z_i\cup Y_i$, where $Y_i\subseteq V(G)\setminus Z_i$ is the set of white vertices that can be forced in the next round if $Z_i$ is the set of the blue vertices. Then the propagation time of $Z$, denoted $pt(G,Z)$, is defined to be $t$. The propagation time of the graph $G$ is then given by $pt(G)=\min_{Z} pt(G,Z)$, where the minimum is taken over all zero forcing sets $Z$ of cardinality $z(G)$. Propagation time for zero forcing has been studied in \cite{HHKMWY2012}.

\subsection{Probabilistic zero forcing}

Zero forcing was initially formulated to bound a problem in linear algebra known as the min-rank problem \cite{AIM2008}. In addition to this application to mathematics, zero forcing also models many real-world propagation processes. One specific application of zero forcing could be to rumor spreading, but the deterministic nature of zero forcing may not be able to fit the chaotic nature of real-life situations. As such, probabilistic zero forcing has also been proposed and studied where blue vertices have a non-zero probability of forcing white neighbors, even if there is more than one white neighbor. More specifically, given a graph $G$, a set of blue vertices $Z$, and vertices $u\in Z$ and $v\in V(G)\setminus Z$ such that $uv\in E(G)$, in a given time step, vertex $u$ will force vertex $v$ to become blue with probability
\[
\mb{P}(u\text{ forces }v)=\frac{|N[u]\cap Z|}{\deg(u)}.
\]

In a given round, each blue vertex will attempt to force each white neighbor independently. If this happens, we may say that the edge $uv$ is forced. A vertex becomes blue as long as it is forced by at least one blue neighbor, or in other words if at least one edge incident with it is forced. Note that if $v$ is the only white neighbor of $u$, then with probability $1$, $u$ forces $v$, so given an initial set of blue vertices, the set of vertices forced via probabilistic zero forcing is always a superset of the set of vertices forced by traditional zero forcing. In this sense, probabilistic zero forcing and traditional zero forcing can be coupled. In the context of rumor spreading, the probabilistic color change rule captures the idea that someone is more likely to spread a rumor if many of their friends have already heard the rumor.

Under probabilistic zero forcing, given a connected graph, it is clear that starting with any non-empty subset of blue vertices will with probability 1 eventually turn the entire graph blue, so the zero forcing number of a graph is not an interesting parameter to study for probabilistic zero forcing. Initially in \cite{KY2013}, the authors studied a parameter which quantifies how likely it is for a subset of vertices to become a traditional zero forcing set the first timestep that it theoretically could under probabilistic zero forcing.

Instead, in this paper, we will be concerned with a parameter that generalizes the zero forcing propagation time. This generalization was first introduced in \cite{GH2018}. Given a graph $G$, and a set $Z\subseteq V(G)$, let $pt_{pzf}(G,Z)$ be the random variable that outputs the propagation time when probabilistic zero forcing is run with the initial blue set $Z$. For ease of notation, we will write $pt_{pzf}(G,v)=pt_{pzf}(G,\{v\})$. The propagation time for the graph $G$ will be defined as the random variable $pt_{pzf}(G)=\min_{v\in V(G)}pt_{pzf}(G,v)$. More specifically, $pt_{pzf}(G)$ is a random variable for the experiment in which $n$ iterations of probabilistic zero forcing are performed independently, one for each vertex of $G$, then the minimum is taken over the propagation times for these $n$ independent iterations.

It is worth mentioning here that probabilistic zero forcing is closely related to the well-studied \emph{push} and \emph{pull} models in theoretical computer science for rumor spreading. In \emph{push}, you start with an infected set of nodes, and at each time step, each infected node chooses a neighbor independently and uniformly at random, and infects the neighbor, if the neighbor is not already infected. Similarly in \emph{pull}, at each time step, each uninfected node chooses a neighbor uniformly at random, it becomes infected if the chosen neighbor was already infected. Finally, the two models can be combined, which is denoted \emph{push\&pull}, where at each time step infected vertices choose a random neighbor to try to infect, and uninfected vertices choose a random neighbor to try to become infected. Similarly to probabilistic zero forcing, the main parameter of interest is the propagation time (or runtime) of these processes. The main differences between \emph{push} and \emph{pull}, and probabilistic zero forcing is that in probabilistic zero forcing, a vertex can force more than one of its neighbors to become blue at every stage, and the probability that a specific blue vertex forces a specific white neighbor increases as more neighbors of the blue vertex become blue. For more information on \emph{push}, \emph{pull} and \emph{push\&pull}, see \cite{DPR2019, MP2016}.

In \cite{GH2018}, the authors studied probabilistic zero forcing, and more specifically the expected propagation time for many specific structures. A summary of this work is provided in the following theorem.  We write $f=O(g)$ if there exists some absolute constant $c$ such that $f
\leq cg$, $f=\Omega(g)$ if $g=O(f)$, and $f=\Theta(g)$ if $f=O(g)$ and $f=\Omega(g)$.
\begin{theorem}{\cite{GH2018}}
Let $n>2$. Then
\begin{itemize}
    \item $\min_{v\in V(P_n)}\E(pt_{pzf}(P_n,v))=\begin{cases}n/2+2/3&\text{ if }n\text{ is even}\\
    n/2+1/2&\text{ if }n\text{ is odd},
    \end{cases}$
    \item$\min_{v\in V(C_n)}\E(pt_{pzf}(C_n,v))=\begin{cases}n/2+1/3&\text{ if }n\text{ is even}\\
    n/2+1/2&\text{ if }n\text{ is odd},
    \end{cases}$
    \item$\min_{v\in V(K_{1,n})}\E(pt_{pzf}(K_{1,n},v))=\Theta(\log n)$,
    \item$\Omega(\log\log n)=\min_{v\in V(K_n)}\E(pt_{pzf}(K_n,v))=O(\log n)$.
\end{itemize}
\end{theorem}

Recently, in \cite{CCGHLOR2019}, the authors used tools developed for Markov chains to analyze the expected propagation time for many small graphs. The authors also showed, in addition to other things, that $\min_{v\in V(K_n)}\E(pt_{pzf}(K_n,v))=\Theta(\log\log n)$ and for any connected graph $G$, $\min_{v\in V(G)}\E(pt_{pzf}(G,v))=O(n)$. This result was then improved very recently in~\cite{NS2019}, where the authors showed that 
$$
\log_2\log_2(n)\leq \min_{v\in V(G)}\E(pt_{pzf}(G,v))\leq \frac{n}2+o(n)
$$ 
for general connected graphs $G$.

The result of most interest to us is from \cite{GH2018}, where in addition to the results mentioned above, the authors also considered the \emph{binomial random graph} $\Gnp$. More precisely, $\Gnp$  is a distribution over the class of graphs with vertex set $[n]$ in which every pair $\{i,j\} \in \binom{[n]}{2}$ appears independently as an edge in $G$ with probability~$p$. Note that $p=p(n)$ may (and usually does) tend to zero as $n$ tends to infinity. We say that $\Gnp$ has some property \emph{asymptotically almost surely} or $\aas$ if the probability that $\Gnp$ has this property tends to $1$ as $n$ goes to infinity.

\begin{theorem}{\cite{GH2018}}
Let $0<p<1$ be constant. Then \aas~we have that
\[
\min_{v\in V(\Gnp)}\E(pt_{pzf}(\Gnp,v))=O((\log n)^2).
\]
\end{theorem}

In addition to this result, the authors in \cite{NS2019} conjectured that for the random graph, a.a.s.\ $\min_{v\in V(\Gnp)}\E(pt_{pzf}(\Gnp,v))=(1+o(1))\log\log n$. 
The main purpose of the current work is to explore probabilistic zero forcing on $\Gnp$ in more detail. Instead of considering the expectation of the propagation time, we will calculate bounds on the propagation time that \aas~hold. We will write $f=o(g)$ or $f\ll g$ if $f/g\to 0$ in the limit, and $f\gg g$ if $g\ll f$. We will write $f\sim g$ if $f=(1+o(1))g$. Our main result is as follows:

\begin{theorem}\label{thm:main}
Suppose that $p = p(n)$ is such that $pn \gg \log n$. Then, for each vertex $v\in V(\Gnp)$, we have that a.a.s.
\begin{align*}
pt_{pzf}(\Gnp,v) &\le (1+o(1)) \Big (\log_2 \log_2 n + \log_3 (1/p) \Big), \qquad \text{ and }\\
pt_{pzf}(\Gnp,v) &\ge (1+o(1)) \max \Big (\log_2 \log_2 n , \log_{4} (1/p) \Big). 
\end{align*}
In particular, if $p = \log^{-o(1)} n$ (for example if $p$ is a constant), then a.a.s.\ 
$$pt_{pzf}(\Gnp,v) \sim \log_2 \log_2 n.$$
On the other hand, if $\log n / n \ll p \le \log^{- O(1)} n$, then a.a.s. 
$$pt_{pzf}(\Gnp,v) = \Theta( \log (1/p) ).$$
\end{theorem}

\subsection{Notation}

We will use the notation $N(v)$ and $N[v]$ to denote the open and closed neighborhoods of the vertex $v$, respectively. Given a set $S\subset V(G)$, we write $N(S)$ for $\left(\bigcup_{v\in S} N(v)\right)\setminus S$. Given two disjoint sets of vertices, $A,B\subseteq V(G)$, we will use $E(A,B)$ to denote the edges with one endpoint in $A$ and one endpoint in $B$, and $e(A,B):=|E(A,B)|$. Similarly, we will write $E(A)$ for the set of edges with both endpoints in $A$, while $e(A):=|E(A)|$.

As mentioned earlier, given two functions $f=f(n)$ and $g=g(n)$, we will write $f=O(g)$ if there exists an absolute constant $c$ such that $f
\leq cg$ for all $n$, $f=\Omega(g)$ if $g=O(f)$, $f=\Theta(g)$ if $f=O(g)$ and $f=\Omega(g)$, and we write $f=o(g)$ or $f\ll g$ if the limit $\lim_{n\to\infty} f/g=0$. In addition, we write $f=\omega(g)$ or $f\gg g$ if $g=o(f)$, and unless otherwise specified, $\omega$ will denote an arbitrarily function that is $\omega(1)$, assumed to grow slowly. We also will write $f\sim g$ if $f=(1+o(1))g$. Through the paper, all logarithms with no subscript denoting the base will be taken to be natural. Finally, as typical in the field of random graphs, for expressions that clearly have to be an integer, we round up or down but do not specify which: the choice of which does not affect the argument.

\section{Probabilistic preliminaries}

In this section we give a few preliminary results that will be useful for the proof of our main result. First, we state a specific instance of Chernoff's bound that we will find useful, then we mention some specific expansion properties that $\Gnp$ has, and finally we mention some helpful coupling results specific to probabilistic zero forcing that we will use in the proof of our main result.







\subsection{Concentration inequalities}

Throughout the paper, we will be using the following concentration inequality. Let $X \in \textrm{Bin}(n,p)$ be a random variable with the binomial distribution with parameters $n$ and $p$. Then, a consequence of Chernoff's bound (see e.g.~\cite[Corollary~2.3]{JLR}) is that 
\begin{equation}\label{chern}
\Prob( |X-\E X| \ge \eps \E X) ) \le 2\exp \left( - \frac {\eps^2 \E X}{3} \right)  
\end{equation}
for  $0 < \eps < 3/2$. Moreover, let us mention that the bound holds for the general case in which $X=\sum_{i=1}^n X_i$ and $X_i \in \textrm{Bernoulli}(p_i)$ with (possibly) different $p_i$ (again, e.g.~see~\cite{JLR} for more details).

\subsection{Expansion properties}

In this paper, we focus on dense random graphs, that is, graphs with average degree $d = p(n-1) \gg \log n$. Such dense random graphs will have some useful expansion properties that hold \aas

\begin{theorem}\label{thm:expansion}
Let $\omega=\omega(n)$ be any function that tends to infinity as $n \to \infty$. Suppose that $d=p(n-1)\ge \omega \log n$. Let $G=(V,E) \in \Gnp$. Then \aas\ the following property holds. Any set $S \subseteq V(G_n)$ of cardinality $s=|S| \leq n/(d \omega)$ satisfies
\[
| N(S) | = sd (1+O(\omega^{-1/2})) \sim sd.
\]
In particular, we get that $\Delta(G) = d (1+O(\omega^{-1/2})) \sim d$ and $\delta(G) = d (1+O(\omega^{-1/2})) \sim d$. 
\end{theorem}

\begin{proof} 
Let $S \subseteq V$, $s=|S|$, and consider the random variable $X = X(S) = |N(S)|$.  We will bound $X$ from above and below in a stochastic sense. There are two things that need to be estimated: the expected value of $X$, and the concentration of $X$ around its expectation. 

It is clear that  
\begin{eqnarray*}
\E X &=& \left( 1 - \left(1- \frac {d}{n-1} \right)^s \right) (n-s) \\
&=& \left( 1 - \exp \left( - \frac {ds}{n} (1+O(d/n)) \right) \right) (n-s) \\
&=& \frac {ds}{n} (1+O(ds/n)) (n-s) \\
&=& ds (1+O(\omega^{-1})).
\end{eqnarray*}

We next use Chernoff's bound, Equation \eqref{chern}, which implies that the expected number of sets $S$ that have $\big| |N(S)| - d|S| \big| > \eps d|S|$ and $|S| \le n/(d\omega)$ is, for $\eps = 2/{\sqrt{\omega}}$, at most 
\begin{align*}
\sum_{s=1}^{n/(d\omega)} 2n^s \exp \left( - \frac {\eps^2 s d}{3+o(1)} \right) 
&=\sum_{s=1}^{n/(d\omega)} 2n^s \exp \left( - \frac {4 s \log n}{3+o(1)} \right)\\
&\leq \sum_{s=1}^{n/(d\omega)} 2n^{-s/(3+o(1))}\\
&\leq 3\cdot2n^{-1/(3+o(1))}+\sum_{s=4}^{n/(d\omega)} 2n^{-s/(3+o(1))}\\
&\leq 6n^{-1/(3+o(1))}+\frac{n}{d\omega}2n^{-4/(3+o(1))}= o(1).
\end{align*}
It follows immediately from Markov's inequality that \aas\ if $ |S| \le n/(d\omega)$, then 
$$
|N(S)| = d |S| (1+O(\omega^{-1/2})),
$$ 
where the constant implicit in $O()$  does not depend on the choice of $S$.  
\end{proof}

\subsection{Useful coupling}

Before we state the lemma, let us recall a standard, but very useful proof technique in probability theory that allows one to compare two random variables or two random processes. Consider two biased coins, the first with probability $p$ of turning up heads and the second with probability $q > p$ of turning up heads. For any fixed $k$, the probability that the first coin produces at least $k$ heads should be less than the probability that the second coin produces at least $k$ heads. However, proving it is rather difficult with a standard counting argument. Coupling easily circumvents this problem. Let $X_1, X_2, \ldots, X_n$ be indicator random variables for heads in a sequence of $n$ flips of the first coin. For the second coin, define a new sequence $Y_1, Y_2, \ldots, Y_n$ such that if $X_i = 1$, then $Y_i = 1$; if $X_i = 0$, then $Y_i = 1$ with probability $(q-p)/(1-p)$. Clearly, the sequence of $Y_i$ has exactly the probability distribution of tosses made with the second coin. However, because of the coupling we trivially get that $X := \sum X_i \le Y:= \sum Y_i$ and so $\Prob(X \ge k) \le \Prob(Y \ge k)$, as expected. We will say that $X$ is (stochastically) bounded from above by $Y$, which we denote by $X \preceq Y$.

We will be using such coupling to simplify both our upper and lower bounds. Indeed, for lower bounds, it might be convenient to make some white vertices blue at some point of the process. Similarly, for upper bounds, one might want to make some blue vertices white.
Given a graph $G$, $S,T\subseteq V(G)$, and $\ell \in \N$, let $A(S,T,\ell)$ be the event that starting with blue set $S$, after $\ell$ rounds every vertex in $T$ is blue.

\begin{lemma}\label{lemma coupling}
For all sets $S_1\subseteq S_2\subseteq V(G)$, $T\subseteq V(G)$, and $\ell\in \N$,
\[
\Prob(A(S_1,T,\ell))\leq \Prob(A(S_2,T,\ell)).
\]
\end{lemma}

\begin{proof}
Let us imagine running two instances of the probabilistic zero forcing process simultaneously, one with initial blue set $S_1$ and the other with initial blue set $S_2$. The process with initial blue set $S_1$, which we will call the first process, will proceed purely at random, while the process with initial blue set $S_2$, henceforth called the second process, will be coupled with the first process.

More precisely, our goal is to show that the two processes can be coupled in such a way that the set of blue vertices in the first process is always a subset of the set of blue vertices in the second process. Once this is achieved, the claim follows immediately. 

Clearly, since $S_1 \subseteq S_2$, the desired property initially holds. Suppose that in the first process, a blue vertex $v \in S_1 \subseteq S_2$ is adjacent to a white vertex $w \notin S_1$. Then, $v$ forces $w$ to become blue with probability $p := |N[v] \cap S_1|/\deg(v)$. Note that $v \in S_2$ so $v$ is also blue in the second process. If $w$ is blue in the second process, then there is nothing to do. Otherwise, we couple the process as follows. If $w$ becomes blue in the first process, then it also becomes blue in the second process. If $w$ stays white in the first process, then $w$ becomes blue in the second process with probability $(q-p)/(1-p)$, where $q := |N[v] \cap S_2|/\deg(v) \ge p$. As a result, $w$ becomes blue in the second process with probability $p + (1-p) \cdot (q-p)/(1-p) = p$, as required. Finally, if $v$ is blue in the second process but it is white in the first one (that is, $v \in S_2 \setminus S_1$), then $w$ becomes blue in the second process with probability $q$ wheres $v$ has no influence on $w$ in the first process. We repeat this argument in each round to get that every time a vertex is forced to become blue in the first process, we will force the same vertex in the second process, unless it is already blue.
\end{proof}

\subsection{Alternative Forcing Processes}\label{Section alternative forcing}
In addition to the above techniques, it will also be useful to consider 
a more general way in which we may augment the forcing process without compromising
our ability to prove lower bounds on the propagation time.

Let us suppose that a subset $B_0$ of the graph $G=(V,E)$
is initially selected to be blue, and a subset $\scr{I} \subseteq \mb{N}_0$
is decided upon before the process begins. The forcing process is then
started, and is allowed to continue until we reach round $i$, where $i$ is the least element of  $\scr{I}$. 
At this point, suppose $B_i$ denotes the blue vertices of $G$ and for each $e \in E$, consider the event in which $e$ is forced in round $i+1$ (this event can occur only if $e \in E(B_{i}, V \setminus B_i)$). Let $Q_{e}^{i}$ denote the probability that this event occurs,
and choose $\til{Q}_{e}^{i}$ such that $Q^i_e \le \til{Q}^{i}_e \le 1$ for each $e \in E(B_{i}, V \setminus B_i)$, and $\til{Q}_{e}^{i}=0$ for $e\in E\setminus E(B_{i}, V \setminus B_i)$ (here
$\til{Q}^i_e$ is a random variable which depends on the process up until time $i$). Define an \textit{alternative forcing
rule} at this time, where each edge $e \in E$ is instead independently forced with probability $\til{Q}^{i}_e$.
When an edge is successfully forced in this framework, any of its remaining white endpoints are turned blue.
After this \textit{alternative forcing step} is performed, the process continues
up until the second smallest $i^* \in \scr{I}$. At this point, step $i^* +1$ is executed in the same manner
as step $i+1$, and the process then continues as in the original framework. We then continue in this way until the entire graph is blue.
Let us refer to a random process defined in this way for index set $\scr{I}$ and random variables
$(\til{Q}^{i}_e)_{i \in \scr{I},e\in E}$ as an \textit{alternative forcing process}.



We may couple an alternative forcing process with the standard forcing process in such a way that the blue vertices in the alternative process always contain those of the original process.
More formally, suppose that a forcing process is started with initial blue vertices $S \subseteq V$, and let $A(S,T,\ell)$ denote the event in which $T \subseteq V$ is colored blue after $\ell \ge 0$ steps. 
If an alternative forcing process $( \scr{I}, ( \til{Q}^{i})_{i \in \scr{I},e\in E})$ is also started at $S$, then let $\til{A}(S,T,l)$ denote the event in which $T$ is colored blue after $\ell$ steps. Under these conditions, the following result holds:

\begin{lemma}\label{lem:alternative_forcing_coupling}
For any alternative forcing process and any subsets $S, T \subseteq V$, 
\[
    \mb{P}( A(S,T,\ell) ) \le \mb{P}( \til{A}(S,T,\ell))
\]
for each $\ell \ge 0$.
\end{lemma}

\section{Upper Bound}\label{sec:upper_bound}

This section is devoted to prove the upper bound in the main result, Theorem~\ref{thm:main}.

\begin{theorem}\label{thm:upper_bound}
Suppose that $p = p(n)$ is such that $pn \gg \log n$. Then for each $v\in V(\Gnp)$ we have that a.a.s.
\[
pt_{pzf}(\Gnp,v) \le (1+o(1)) \log_2 \log_2 n + (1+o(1)) \log_3 (1/p).
\]
\end{theorem}

Before we move to analyzing the process, let us mention how we are going to apply Theorem~\ref{thm:expansion}. This is a standard technique in the theory of random graphs but it is quite delicate. We wish to use the expansion properties guaranteed a.a.s.\ in Theorem~\ref{thm:expansion}, but we also wish to avoid working in a conditional probability space.

Thus, we will use an unconditioned probability space, but we will provide an argument that assumes we have the expansion properties of Theorem~\ref{thm:expansion}. Since these properties hold a.a.s., the measure of the set of outcomes in which our argument does not apply to is $o(1)$, and thus can be safely excised at the end of the argument. 

\begin{proof}[Proof of Theorem~\ref{thm:upper_bound}]
Fix $\omega = \omega(n)$ to be a function that tends to infinity arbitrarily slowly so that some inequalities below hold. Let $p=p(n)$ be such that $d=p(n-1) \ge \omega \log n$.

\medskip

\textbf{Phase 1:} We start the process with an arbitrary vertex $v \in V(G)$ and we expose all edges from $v$ to the rest of the graph. By Theorem~\ref{thm:expansion}, we may assume that $\deg(v) \sim d$. This phase lasts 
\[
t_1 := \frac {\log \log n}{\log \log \log n} = o(\log \log n)
\] 
rounds. We will prove that \aas\ at the end of this phase at least 
\[
b_1 := t_1 \left( 1 - \frac {\log \log \log n}{\log \log n} \right) \sim t_1
\] 
neighbors of $v$ are blue. (Let us mention that the choice of $t_1=t_1(n)$ is rather arbitrary. Any function tending to infinity as $n \to \infty$ would work. On the other hand, it is convenient to have $t_1 = o(\log \log n)$ so that the length of this phase is negligible compared to the total length.)

Fix any $w \in N(\{v\})$. The probability that $w$ is white at the end of Phase 1 is at most
\[
\left( 1 - \frac {1}{\deg(v)} \right)^{t_1} \le \exp \left( - \frac {t_1}{\deg(v)} \right) = 1 - \frac {t_1}{\deg(v)} ( 1 + O( t_1 / d )) \le 1 - q_1
\]
for 
\[
q_1 :=  \frac {t_1}{\deg(v)} \left( 1 - \frac{\log \log n}{\log n} \right),
\]
as $d \gg \log n$. Hence, the number of neighbors of $v$ that are blue at the end of Phase~1 can be stochastically lower bounded by a random variable $X_1 \in \Bin(\deg(v), q_1)$ with $\E [X_1] = t_1 ( 1 - \log \log n / \log n )$. After applying Chernoff's bound (\ref{chern}) with 
\[
\eps = \frac { (\log \log \log n)^{2/3} }{ (\log \log n)^{1/2} }
\]
we get that \aas\ $X_1 = \E[X_1] (1+O(\eps)) = t_1 (1+O(\eps)) \ge b_1$. (Note that $\eps^2 t_1 = (\log \log \log n)^{1/3} \to \infty$ as $n \to \infty$.) 

\medskip

\textbf{Phase 2:} We start this phase with $b_1$ blue vertices. Indeed, we know that after Phase~1, \aas\ we have at least $b_1$ blue vertices, and via Lemma~\ref{lemma coupling}, we may assume that we have exactly $b_1$ vertices while still claiming an upper bound.

We will show that \aas, in each round of Phase~2, the number of blue vertices increases by at least a multiplicative factor of $A := 3(1-\omega^{-1/4})$. This phase ends when the number of blue vertices exceeds 
\[
b_2 := \frac {n}{d\omega}.
\]
Let us note that for very dense graphs it might happen that $b_2 < b_1$ and if this happens, then this phase actually does not occur. For sparser graphs, \aas\ this phase lasts at most
\begin{align*}
t_2 := \log_A \left( \frac {b_2}{b_1} \right) &= \log_A \left( \frac {n \log \log \log n}{d \omega \log \log n} (1+o(1)) \right) \le \log_A (1/p)\\
&= \frac {\log_3(1/p)}{\log_3 A} = (1+O(\omega^{-1/4})) \log_3 (1/p) \sim \log_3 (1/p)
\end{align*}
rounds, provided that $\omega = \omega(n)$ tends to infinity sufficiently slowly.

Suppose that at the beginning of some round, blue vertices form set $S$ of size $s=|S|$, where $b_1 \le s < b_2$. By Theorem~\ref{thm:expansion}, we may assume that $|N(S)| = ds (1+O(\omega^{-1/2})) \sim ds$ and that $\delta(G) = \Delta(G) (1+O(\omega^{-1/2})) = d (1+O(\omega^{-1/2})) \sim d$. Fix any $w \in N(S)$ and let $v$ be a neighbor of $w$ in $S$ (if $w$ has more than one neighbor in $S$, pick one of them arbitrarily). Since $v$ is not only blue but it has at least one blue neighbor (note that the process guarantees that $S$ induces a connected graph), $v$ forces $w$ to become blue with the probability at least 
\[
\frac {2}{\deg(v)} = \frac {2}{(1+O(\omega^{-1/2}))d} = \frac {2 + O(\omega^{-1/2}) }{d} > \frac {2 - \omega^{-1/3}}{d} =: q_2.
\]

As a result, the number of vertices in $N(S)$ that become blue at the end of this round can be stochastically lower bounded by the random variable $X_2 \in \Bin(|N(S)|,q_2)$. Since $|N(S)|=ds(1+O(\omega^{-1/2}))$, we have $\E [X_2] = 2s ( 1 + O(\omega^{-1/3}))$. It follows from Chernoff's bound~\eqref{chern} applied with $\eps = \omega^{-1/2}$ that $X_2 = 2s ( 1 + O(\omega^{-1/3}))$ with probability $1 - \exp ( - \Theta( s/\omega ))$. We will say that the round is bad if at the end of it the number of blue vertices is less than $As = 3s(1-\omega^{-1/4})$, that is if $s+X_2<As$, which only occurs when $X_2 < 2s(1-\omega^{-1/4})$.

It is worth noting here that we are performing our calculations here on the assumption that no bad rounds have occurred previously. This is justified by the observation that the probability that some round in Phase 2 is bad is at most
\[
\sum_{i \ge 0} \exp \left( - \Theta \big( A^i b_1 / \omega \big) \right) \le \sum_{i \ge 0} 2^{-i} \exp \left( - \Theta \big( b_1 / \omega \big) \right) = 2 \exp \left( - \Theta \big( b_1 / \omega \big) \right) = o(1),
\]
provided that $\omega = \omega(n)$ tends to infinity sufficiently slowly. Hence, \aas\ the second phase ends with $b_2$ blue vertices in at most $t_2$ rounds (again, provided $b_2 > b_1$; otherwise there is no Phase 2).

At this point, it will be useful to define the set $S_2$ as the set of vertices that were initially blue in the last round of Phase $2$. This definition will be useful once we begin Phase $4$.

\medskip

\textbf{Phase 3:} Suppose first that $b_2 \ge b_1$, that is, 
\[
d= \frac {n}{b_2 \omega} \le \frac {n}{b_1 \omega} \sim \frac {n \log \log \log n}{\omega \log \log n}.
\]
The argument for very dense graphs can be easily adjusted, and we will come back to this once we deal with sparser graphs.

We start this phase with $b_2 = n/(d\omega)$ blue vertices that form the set $S$.  By Theorem~\ref{thm:expansion}, we may assume that $|N(S)|=db_2 (1+O(\omega^{-1/2})) = (n/\omega) (1+O(\omega^{-1/2})) \sim n/\omega$. This phase lasts 
\[
t_3 := \frac {\log \log n}{\log \log \log n} = o(\log \log n)
\] 
rounds. We will prove that \aas\ at the end of this phase at least 
\[
b_3 := \frac {n \log \log n}{d (\log \log \log n)^2} \gg \frac {n}{d}
\] 
vertices of $N(S)$ are blue. 

Arguing as in Phase 2, in each round, each white vertex $w$ in $N(S)$ becomes blue with probability at least $q_2 = (2-\omega^{-1/3})/d$. Hence, the probability that $w$ is blue at the end of this phase is at least
\begin{align*}
1-(1-q_2)^{t_3} & \ge 1 - \exp (-q_2t_3) = q_2t_3 (1+O(q_2t_3)) \\
& = \frac {2t_3}{d} (1+O(\omega^{-1/3})) \ge \frac {2t_3}{d} (1 - \omega^{-1/4}) =: q_3,
\end{align*}
provided that $\omega = \omega(n)$ tends to infinity sufficiently slowly. As it was done earlier, the number of vertices in $N(S)$ that become blue at the end of this phase can be lower bounded by random variable $X_3 \in \Bin(|N(S)|,q_3)$ with 
\[
\E [X_3] = |N(S)| \cdot q_3 \sim \frac {n}{\omega} \cdot \frac {2t_3}{d} \gg \frac {n \log \log n}{d (\log \log \log n)^2},
\]
and Chernoff's bound~(\ref{chern}) implies the conclusion. 

Let $S_3$ denote the initial set of blue vertices in the very last round of Phase $3$. It is worth noting here that the only edges of $\Gnp$ that are exposed at the end of this phase are the edges within $S_3$ and edges between $S_3$ and $V(G) \setminus S_3$. This will be important in the next phase.

\medskip

Finally, let us discuss how to deal with very dense graphs. If $d \ge 2 n / \log \log \log n$, then $b_1 > b_3$ and so there is nothing to do: there are more than $b_3$ blue vertices at the end of Phase 1 and so there is no Phase 2 nor Phase 3, so we just proceed immediately to Phase 4. For $d$ such that
\[
\frac {n \log \log \log n}{\omega \log \log n} \sim \frac {n}{b_1 \omega} < d < \frac {2n}{\log \log \log n},
\]
there is no Phase 2 ($b_1 > b_2$) but there is Phase 3 ($b_1 < b_3$). This time, instead of starting Phase 3 with all blue vertices, we select (arbitrarily, while retaining connectivity of the blue subgraph) any subset of $b_2$ blue vertices and proceed with the argument as before. Lemma \ref{lemma coupling} implies that we may give away these vertices while still maintaining a strict upper bound on the entire length of the process.

\medskip

\textbf{Phase 4:} Up until this point, we have done our calculations assuming that the blue vertices which are performing forces have only one blue neighbor. Heuristically, we should not have lost too much with this assumption since thus far we have had only a negligible number of blue vertices in each phase, and thus expect few edges within the blue subgraph. In this phase, the set of blue vertices will grow to be large enough that we expect many edge in the blue subgraph. Our analysis will exploit this fact to provide better bounds on how fast the number of blue vertices grows with every step.

This phase consists of some number of rounds that are going to be indexed with $i \in \N$.
At the beginning of this phase, we will partition the blue vertices into two sets, $Y_0$ and $Y_1$, where
\[
Y_0=\begin{cases} \{v\}&\text{ if }b_1>b_2,b_1>b_3\\
S_2&\text{ if }b_2>b_1,b_2>b_3\\
S_3&\text{ if }b_3>b_1,b_3>b_2,
\end{cases}
\]
and $Y_1$ is the set of blue vertices that are not in $Y_0$. Then, at the beginning of Round $i$, let $Y_i$ be the set of vertices that were turned blue in the previous round. Thus the $Y_j$'s ($0 \le j \le i$) partition the blue vertices at the beginning of Round $i$. An important property is that the only edges that are exposed at this point are the edges with at least one endpoint in $Y_{\leq i-1}$, where  $Y_{\le \ell} := \bigcup_{k=0}^{\ell}Y_k$.

Let us concentrate on a given Round $i$, $i \in \N$. Let $y_i=|Y_i|$ and suppose that 
\[
y_i \ge t^{2^{i-1}} p^{-1} 8^{2-2^i} \ge t/p
\qquad \text{ and } \qquad
y_i \le \sqrt{\frac {n}{p\omega}}.
\]
Let us label vertices of $Y_i$ as $v_1, v_2, \ldots, v_{y_i}$. Our fist task is to identify a set $Z_i \subseteq Y_i$ of blue vertices with at least $y_ip/3$ neighbors in $Y_i$. These vertices have a strong forcing power; it will be convenient to use them to control the number of white vertices that become blue at the end of this round. In order to simplify the argument and keep events independent, let us partition $Y_i$ into $Y_i^- = \{v_\ell : \ell \le y_i/2\}$ and $Y_i^+ = Y_i \setminus Y_i^-$. Now, for any vertex $v \in Y_i^+$, we expose edges from $v$ to $Y_i^-$ and put $v$ into $Z_i$ if $X_v$, the number of neighbors in $Y_i^-$, is at least $y_ip/3$. Note that $X_v \in \Bin(\lfloor y_i/2 \rfloor, p)$ with $\E[X_v] \sim y_ip/2 \to \infty$ so $v \in Z_i$ with probability at least $1/2$ (in fact, it tends to 1). Hence, $|Z_i|$ can be stochastically lower bounded by $\Bin(\lceil y_i/2 \rceil,1/2)$ with expectation at least $y_i/4$. It follows from Chernoff's bound~(\ref{chern}) that $|Z_i| \ge y_i/5$ with probability $1 - \exp( -\Theta(y_i))$. If $|Z_i| < y_i/5$, then we say that this round fails and we finish the process prematurely (later we will show that a.a.s.\ no round will fail). 

Our next task is to estimate the number of white vertices (that is, vertices in $V \setminus Y_{\le i}$) that are adjacent to at least $y_ip/6$ vertices in $Z_i$; we will call them \emph{good}. Fix $w \in V \setminus Y_{\le i}$. The expected number of neighbors of $w$ in $Z_i$ is equal to $|Z_i|p \ge y_ip/5$. Hence, by Chernoff's bound~(\ref{chern}), $w$ is good with probability at least 1/2 (as before, in fact, it tends to 1). Hence, the number of good vertices is lower bounded by $\Bin(|V \setminus Y_{\le i}|, 1/2)$ with expectation asymptotic to $n/2$, since $\sum_{0 \le \ell \le i} y_\ell = o(n)$. Hence, with probability at least $1-\exp(-\Theta(n))$, there are at least $n/3$ good vertices. (We have a lot of room in the argument here.) If the number of good vertices is less than $n/3$, then we say that this round fails and we finish the process prematurely. 

Our final task is to estimate how many good vertices become blue at the end of Round $i$. Fix any good vertex $w$. Since each neighbor of $w$ in $Z_i$ forces $w$ to become blue with probability $(y_ip/3)/d(1+o(1))$, $w$ stays white with probability at most
\begin{align*}
\left( 1 - \frac {y_ip/3}{d(1+o(1))} \right)^{y_ip/6} &\le \exp \left( - \frac {y_i^2p^2}{19d} \right) \le \exp \left( - \frac {y_i^2p}{19n} \right) \\
&= 1 - \frac {y_i^2p}{19n} (1+O(y_i^2p/n)) \le 1 - \frac {y_i^2p}{20n},
\end{align*}
since $y_i^2p/n \le (n/(p\omega))p/n = 1/\omega = o(1)$. Hence, the expected number of good vertices that become blue is at least $(n/3)(y_i^2p/(20n)) = y_i^2p/60$. By Chernoff's bound~(\ref{chern}), with probability at least $1- \exp( - \Theta(y_i^2p))$ there are at least $y_{i+1} := (y_i/8)^2p$ new blue vertices that form $Y_{i+1}$. As always, we say that the round fails if there are not enough new blue vertices and we stop the process prematurely. 

Recall that 
\[
y_i \ge t^{2^{i-1}} p^{-1} 8^{2-2^i}  
\]
and so 
\[
y_{i+1} = \left( \frac {y_i}{8} \right)^2 p \ge t^{2^{i}} p^{-1} 8^{2-2^{i+1}} .
\]
We will run this phase for at most
\[
t_4 := \log_2 \log_{t/64} n \le \log_2 \log_2 n
\]
steps. Phase 4 finishes prematurely with probability at most
\begin{align*}
\sum_{i\geq 1}  \left( \exp ( - \Theta(y_i) ) + \exp ( - \Theta(n) ) + \exp ( - \Theta(y_i^2 p) )\right) &= \sum_{i \ge 1} \exp ( - \Theta(t^{2^{i-1}} p^{-1}) ) \\
&\leq t_4\exp ( - \Theta(t p^{-1}) )\\
&= \exp ( - \Theta(b_3) ) = o(1).
\end{align*}
We will stop phase 4 either once there are more than 
\[
b_4 := \sqrt{\frac {n}{p\omega}}
\]
new blue vertices, or $t_4$ steps have passed. Since at the end of Round $i$, there are $y_{i+1} \ge (t/64)^{2^i} / p \ge (t/64)^{2^i}$ new blue vertices, \aas\ we have $b_4$ new vertices before $t_4$ total rounds have elapsed, so \aas\ Phase 4 ends once we have $b_4$ new blue vertices. These new vertices will be able to force the rest of the graph blue in the next two rounds. 
\medskip

\textbf{Two Last Rounds (Phase 5):} Recall that at the beginning of the first round of Phase 5 there are more than $b_4 = \sqrt{n/ (p\omega)}$ new blue vertices. We  select any subset of $b_4$ vertices (arbitrarily) and the analysis above implies that \aas\ at least $(b_4/8)^2 p = n / (64 \omega)$ white vertices become blue; let us put them into a  set $Y$.

At the beginning of the final round of the whole process, similarly to the previous phase, \aas\ we can find a set $Z\subset Y$ with $|Z|\geq |Y|/5$ with every vertex in $Z$ having at least $|Y|p/3$ neighbors. Each white vertex expects at least $|Y|p/5 = (n / 64 \omega) p/5 \gg \log n$ neighbors in $Z$ (since $d = p(n-1) \gg \log n$ and $\omega=\omega(n)$ is tending to infinity arbitrarily slowly). By Chernoff's bound~(\ref{chern}), the expected number of white vertices that are \emph{not} good is at most $n \exp ( - \Theta( |Z|p )) \le n \exp ( - 2 \log n) = n^{-1} = o(1)$ and so, by Markov's inequality, \aas\ all white vertices are good. Finally, arguing as before, the expected number of  good (white) vertices that stay white is at most 
\begin{eqnarray*}
n \exp \left( - \Theta \left( \frac {|Z|^2 p^2}{d} \right) \right) &=& n \exp \left( - \Theta \left( \frac {(np)^2}{d \omega^2} \right) \right) = n \exp \left( - \Theta \left( \frac {d}{\omega^2} \right) \right)  \\ 
&=& n \exp \left( - 2 \log n \right) = n^{-1} = o(1).
\end{eqnarray*}
It follows that \aas\ all vertices become blue and the process is over.

\medskip

Adding up the total time, this analysis shows that a.a.s.\ we have
\begin{align*}
pt_{pzf}(\Gnp,v)&\leq t_1+t_2+t_3+t_4+2\\
&\leq o(\log\log n)+(1+o(1))\log_3(1/p)+o(\log\log n)+\log_2\log_2 n+2\\
&=(1+o(1))\log_2\log_2 n+ (1+o(1))\log_3 (1/p),
\end{align*}
as claimed.
\end{proof}

\section{Lower Bound}

This section is devoted to prove the lower bound in the main result, Theorem~\ref{thm:main}. For simplicity, we independently consider sparse and dense random graphs, starting from the dense case that is easier to deal with. Let us point out that Theorem~\ref{thm:lower_bound_dense} yields the bound of $(1+o(1))\log_{2}\log_2 n$ whereas Theorem~\ref{thm:lower_bound_sparse} yields the bound of $(1+o(1)) \log_{4}(1/p) = (1/2+o(1)) \log _2 (1/p) \ge (1+o(1)) \log_2 \log_2 n$ as it holds for $1/p \ge \log^2 n$. As a result, both theorems imply a general bound of $(1+o(1)) \max ( \log_2 \log_2 n, \log_4 (1/p) )$, as claimed in the main result. Let us also point out that the claimed lower bound holds with probability $1-o(n^{-1})$ and so, in fact, we get that a.a.s.\ $pt_{pzf}(\Gnp) \ge (1+o(1))\log_{2}\log_2 n$.

\begin{theorem} \label{thm:lower_bound_dense}
Suppose that $p=p(n) \ge 1/\log^{2} n$. Then, for each vertex $v\in V(\Gnp)$, the following bound holds with probability $1-o(n^{-1})$:
$$
pt_{pzf}(\Gnp,v) \ge (1+o(1))\log_{2}\log_2 n.
$$
\end{theorem}

\begin{proof}
Let $\omega=\omega(n)$ be any function tending to infinity (slowly enough) as $n \to \infty$. Since we aim for a lower bound, by Lemma~\ref{lemma coupling}, we may assume that we begin with a subset $Y_{0} \subseteq V$ of size $y_0 = b_{0}:= \omega \log n / p$, which consists of the vertices of the graph which are initially blue, including the specified vertex $v$. The forcing process is then started, and for each $i \in \mathbb{N}$ we denote $Y_{i} \subseteq V$ as the vertices of $\Gnp$ which are turned blue in round $i$.
If we fix $i \ge 0$, then $Y_{ \le i}:= \bigcup_{j=0}^{i}Y_{j}$ consists of the blue vertices
of $\Gnp$ after the first $i$ rounds. 
Finally, let $b_{i}:= |Y_{\le i}|$ and $y_{i}:=|Y_{i}|$. 
We may define the stopping time $\tau \ge 0$ to be the
first $i \ge 0$ such that $Y_{ \le i} = V$. Our goal is to show that with probability $1-o(n^{-1})$ we have $\tau \ge (1 + o(1))\log_{2}\log_2 n$.

In order to achieve this bound, we must control the number of white vertices which can be
forced in a given round. As a result, we need to be careful in which order we expose edges of $\Gnp$. We will preserve the following property at the beginning of round $i$ for each $i \ge 1$: 
\begin{enumerate} 
    \item[(P0)] the edges between $V \setminus Y_{\le i-1}$ and $Y_{i-1}$ are \emph{not} exposed yet.
\end{enumerate}
Indeed, it will be possible since the forcing at step $i$ has not occurred yet and the pairs of vertices that are involved were white in the previous step (clearly, edges between two white vertices cannot contribute to any forcing). On the other hand, the edges between $V \setminus Y_{\le i-1}$ and $Y_{\le i-2}$ are already exposed (when $i \ge 2$). We now expose the edges between $V \setminus Y_{\le i-1}$ and $Y_{i-1}$ and we check for the following property:
\begin{enumerate}
    \item[(P1)] $\deg_{Y_{i-1}}(v) \le 2 \, p \, y_{i-1}$ for all $v \in V \setminus Y_{\le \, i-1}$. 
\end{enumerate}
Finally, the forcing takes place and at the end of this round we investigate the following property:
\begin{enumerate}
    \item[(P2)] $y_{i} \le 3 \, b_{i-1}^{2}$. 
\end{enumerate}

We say that round $i$ is \textit{good} provided the two properties (P1) and (P2) are satisfied.
In fact, since we aim for a lower bound, by Lemma~\ref{lemma coupling}, we may assume that $y_j = 3 \, b_{j-1}^{2}$; that is, at the end of each round we may turn some additional vertices blue to satisfy this equality. 

Set 
\begin{align*}
t_{F} &:= \log_2 \left( \frac {\log_2 (n^{1/3})}{\log_2 (4 \omega \log n / p)} \right)\\
&=  \log_{2} \log_2 n - \log_2 \log_{2}( 4 \omega \log n / p) - O(1) \\
&= \log_2 \log_2 n - \log_2 \log_2 \log_2 n - O(1) = (1+o(1)) \log_2 \log_2 n,
\end{align*}
as $p \ge 1/\log^{2} n$.
Suppose that the first $t_{F}$ rounds are good. Under this assumption, we have that
\begin{align*}
b_{i} = b_{i-1} + y_i &\le (3+o(1)) b_{i-1}^2\\ &\le 2^2 b_{i-1}^2 \le \ldots \le 2^{2^{i}+2^{i-1}+\ldots+2} b_0^{2^i} \le 2^{2\cdot 2^i} b_0^{2^i} = \left( 4 \omega \log n / p \right)^{2^i}
\end{align*}
for all $1 \le i \le t_{F}$. This implies that $b_{t_F} \le n^{1/3} < n$, and so $\tau > t_{F}$ which yields the desired claim. We shall now prove that with probability $1-o(n^{-1})$ the first $t_{F}$ rounds are good, which will complete the proof.

Let us fix $0 \le i < t_{F}$, and assume that the first $i$ rounds are good (Note that there is nothing to assume when $i=0$; that is, when we begin the first round). As we previously noted,
we may assume that the edges between $Y_{i}$ and $V \setminus Y_{ \le i}$ are unexposed at this time--- see property (P0). Under these assumptions,
we will estimate the probability that round $i+1$ is good as well.

Let us start with property (P1). Observe that for each $v \in V \setminus Y_{\le i}$, we have that $\mathbb{E} \deg_{Y_i}(v) = y_{i} p$.
Since the first $i$ rounds are good, $y_{i} p \ge y_{0}p= \omega \log n$.
By Chernoff's bound (\ref{chern}), this implies
\[
    \mathbb{P}( |\deg_{Y_i}(v) - y_{i}p| \ge y_{i}p) \le 2\exp(-y_{i}p/3) \le 2 n^{-\omega/3}
\]
for each $v \in V \setminus Y_{\le i}$. We may therefore use the union bound to conclude that
\begin{equation*}
    \mathbb{P}( \cup_{v \in V \setminus Y_{\le i}} | \deg_{Y_i}(v) - y_{i}p | \ge y_{i}p) \le n^{-\omega/3 +1} = n^{-\Theta(\omega)} = o((nt_F)^{-1}).
\end{equation*}
As a result, with probability $1-o((nt_F)^{-1})$, every vertex of $v \in V \setminus Y_{\le i}$ has $\deg_{Y_{i}}(v) \le 2 y_{i} p$ (that is, property~(P1) is satisfied for round $i+1$),
provided the first $i$ rounds are good. 

Let us now assume that property~(P1) holds and move to investigating property~(P2). Consider the probability that $v \in V \setminus Y_{\le i}$ is \textit{not} forced by $u \in Y_{\le i}$ with $uv\in E(\Gnp)$.
Clearly,
\[
    \mathbb{P}(\text{$u$ does not force $v$}) = \left( 1 - \frac{\deg_{Y_{\le i}[u]}}{\deg(u)} \right) \ge \left( 1 - \frac{b_i}{\delta} \right),
\]
where $\delta$ is the smallest degree of $\Gnp$. Thus,
\[
    \mathbb{P}(\text{$Y_{\le i}$ does not force $v$})  \ge \left( 1 - \frac{b_i}{\delta} \right)^{\deg_{Y_{\le i}}(v)} \ge \left(1 - \frac{b_i}{\delta}\right)^{2 b_i p},
\]
as we may assume that $\deg_{Y_{\le i}}(v) \le \sum_{j=0}^{i} 2 y_{i} p = 2 b_{i} p$ in light of property~(P1). 
By Theorem~\ref{thm:expansion}, we may also assume that $\delta \sim  np$. (See the beginning of Section~\ref{sec:upper_bound} for a discussion how this theorem is applied.) 
Note that 
$$
\frac {b_i}{np} \le \frac {b_i}{np} (2b_ip) = \frac {2b_i^2}{n} \le \frac {2b_{t_F}^2}{n} \le \frac {2n^{2/3}}{n} = o(1)
$$
when $i< t_{F}$. Thus,
\begin{align*}
    \mathbb{P}(\text{$Y_{\le i}$ does not force $v$}) & \ge \left(1 - \frac{b_i}{n p}\right)^{2 b_i p} = \exp\left(-\frac{2b^2_{i}p}{np} (1+o(1)) \right) \\
                                                & = 1 - \frac{2b^2_{i}}{n} (1+o(1)).
\end{align*}
We may therefore conclude that
\[
    \mb{P}( \text{$Y_{\le i}$ forces $v$}) \le (1+ o(1)) \frac{2b^{2}_i}{n},
\]
for all $v \in V \setminus Y_{\le i}$. Let us now consider $Y_{i+1}$, the vertices forced by $Y_{\le i}$
in a single round. If we assume that the above claims hold, then we have that
\[
    \mb{E} |Y_{i+1}| = \sum_{v \in V\setminus Y_{\le i}} \mb{P}( \text{$Y_{\le i}$ forces $v$}) \le (1 + o(1)) 2b^{2}_i.
\]
Now the first $i$ rounds are assumed to be good, so we know that 
$b_{i} \ge b_{0} = \omega \log n /p$.
We may therefore use Chernoff's bound (\ref{chern}) to observe that

\begin{align*}
    \mb{P}( |Y_{i+1}| \ge 3 b_{i}^{2}) &\le \exp\left(\frac{- (2+o(1)) b_i^{2}}{3}\right)  \le \exp\left( -\frac{\omega^{2}\log^{2}n}{3p^{2}}\right) \le \exp\left( - \frac{\omega^{2}\log^{2}n}{3}\right)       \\
                                       & = n^{-\omega^2 \log n / 3} = o((nt_F)^{-1}).
\end{align*}
As a result, we have that with probability $1-o((nt_F)^{-1})$, $|Y_{i+1}| \le 3 \, b_i^{2}$, and so the total number of
blue vertices in $Y_{i+1}$ is of the desired amount (property~(P2)). 

Combining the two concentration results, we get that if the first $i$ rounds are good, then with probability $1-o((nt_F)^{-1})$ round $i+1$ is good as well.
We remark that when $i=0$, the above analysis shows that with probability $1-o((nt_F)^{-1})$ round $1$ is good.
In order to complete the proof, observe that 
\begin{align*}
\mb{P}(\cup_{i=1}^{t_{F}}\text{Round $i$ is bad}) &= \sum_{i=1}^{t_F} \mb{P}(\text{Round $i$ is bad and the earlier rounds are good}) \\
&\le t_F \cdot o((nt_F)^{-1}) = o(n^{-1}),
\end{align*}
and the proof is complete.
\end{proof}

\bigskip

Let us now move to the sparse case. Before we prove our lower bound
in this regime, let us discuss how the techniques used here differ from those seen in the previous section.
Suppose that $Y_0 \subseteq V$ is initially colored blue, and a forcing process
is begun. Specifically, let $Y_{i}$ denote the blue vertices of $\Gnp$ which are forced in round
$i$ for each $i \ge 1$. As before, we use $Y_{\le i}$ to denote the collection of blue vertices
after $i \ge 0$ rounds. 

In the previous arguments, we have been careful to ensure that the edges 
with one endpoint in $V \setminus Y_{\le i}$
and one in $Y_{i}$ remain unexposed by the time we consider the vertices forced in round $i+1$, namely $Y_{i+1} \subseteq V \setminus Y_{\le i}$. This allows us to ensure that the edges between $Y_{\le i}$ and $V \setminus Y_{\le i}$ are distributed as $\Bin( |Y_{\le i}||V \setminus Y_{\le i}|,p)$,
which proves convenient in our computations. In particular, we are able to guarantee that each vertex $v \in V \setminus Y_{\le i}$
has $\mb{E}( \deg_{Y_{\le i}}(v)) \rightarrow \infty$, which makes proving concentration via Chernoff's bound amenable. 

While such techniques work when $p(n) \ge 1 / \log n$, 
this is a corollary of range of $|Y_{\le i}|$ which we concern ourselves with in the above argument; namely, when $|Y_{\le i}| \gg \log n /p$.
When $p(n) < 1/ \log n$, we are instead interested in analyzing $|Y_{\le i}|$ up until the point at which 
$|Y_{\le i}| = O( 1/p)$. As a result, we do \textit{not} have sharp degree concentration throughout the range
we must analyze. Specifically, if $|Y_{\le i}| = O( 1/p)$, then $\mb{E}( \deg_{Y_{\le i}}(v)) = O(1)$
for each $v \in V \setminus Y_{\le i}$, and so $\deg_{Y_{\le i}}(v)$ does not witness sharp concentration as $n \rightarrow \infty$.

In order to circumvent these issues, we first restrict ourselves to an alternative forcing process (recall that alternative forcing processes were described in Section \ref{Section alternative forcing}) defined as follows. We first work with the index set taken to be all of $\mb{N}_0$. Thus, we shall apply alternative forcing rules throughout every step of the process. If we then consider the round $i \in \mb{N}_0$, $u \in Y_{ \le i}$ forces its white neighbors with probability $\min\{\deg_{Y_{\le i}}[u]/ \til{d}_{L},1\}$, where $d:= (n-1)p$ and $\til{d}_{L}:= ( 1 - \omega^{-1})d$ (here $\omega$ tends to infinity arbitrarily slowly). As $\aas$\ each vertex of $\Gnp$ has degree at least $\til{d}_L$, we  know that this alternative forcing process is valid $\aas$; that is, the conditions defined in Section \ref{Section alternative forcing} are satisfied. As a result, we may apply Lemma~\ref{lem:alternative_forcing_coupling} to couple this alternative forcing process with the standard forcing process for all but $o(1)$ of the instances of $\Gnp$. Specifically, any $\aas$ lower bound on this alternative forcing process will imply an $\aas$ lower bound on the standard forcing process. In what follows, all our results shall be  with respect to this alternative forcing process. We shall use the same terminology as before, as it should be unambiguous regarding which process we are referring to.

Before continuing, let us motivate why it is convenient to work with this alternative forcing rule. If we consider a vertex $v \in V \setminus Y_{0}$ and $u \in Y_{0}$,
then the probability that $u$ forces $v$ in a single round is $p \, ( \min\{1, \deg_{Y_{0}}[u]/ \til{d}_{L}\})$,
as $\deg_{Y_{0}}[u]$ is independent of the event $(u,v) \in \Gnp$ and $\til{d}_{L}$ is a fixed
value. This observation allows us to prove the following lemma easily. The proof of this auxiliary lemma can be found in Section~\ref{sec:auxiliary}.

\begin{lemma}\label{lem:forcing_independence}
Suppose that we are given $\Gnp$ on vertex set $V$, with $Y_{0} \subseteq V$ initially
blue where $|Y_0|=k$ and $Y_0=\{u_1, \ldots ,u_k\}$. For each $u \in Y_0$, denote $Y_1(u)$ as the vertices of $V \setminus Y_0$ which are
forced by $u$ after a single round.
\begin{enumerate} \label{eqn:forcing_independence}
    \item If $S_{u_1},\ldots ,S_{u_k} \subseteq V \setminus Y_0$, then for each $u \in U$ and $v \in V \setminus Y_0 \cup S_{u}$,
    \[
    \mb{P}( (u,v) \in \Gnp \, | \, \cap_{j=1}^{k}  Y_1(u_j) = S_{u_j}) \le p.
    \]
    \item The indicator random variables $\{ \mathbf{1}_{[ (u,v) \in \Gnp]} \}_{ u \in Y_0 , v \in V \setminus Y_0 \cup S_{u}}$ are conditionally independent
    of the event $\cap_{j=1}^{k} \{ Y_1(u_j) = S_{u_j} \}$. 
\end{enumerate}
\end{lemma}

While the above lemma appears as if it can only be applied to the first round of the forcing
process, it can in fact be used through the subsequent rounds as well. To see how this can be done, observe
that once $Y_{1}$ is conditioned upon and all the edges contained within $Y_{\le 1}$ are exposed,
the indicator variables $\{ \mathbf{1}_{[ (u,v) \in \Gnp]} \}_{ u \in Y_{\le 1} , v \in V \setminus Y_{\le 1}}$ are independent and each occur with probability at most $p$. 
Moreover, the edges within $V \setminus Y_{\le 1}$ are independent
of $Y_{\le 1}$, and thus each occur independently with probability $p$.

At this point, let us denote $\scr{G}_1$ as $\Gnp$ conditioned on the vertices $Y_{\le 1}$ and the edges within this set. More precisely, we are conditioning on $Y_1$,
the resulting vertices after one round of probabilistic zero forcing, followed by the edges in $\Gnp$ within $Y_{\le1}$.
Moreover, let us consider another random graph $\scr{G}_2$, whose edges in $Y_{\le 1}$ are the same as that of
$\scr{G}_1$. On the other hand, the edges from $Y_{\le 1}$ into $V \setminus Y_{\le 1}$ and edges completely within $V \setminus Y_{\le 1}$
are each defined to occur independently with probability exactly $p$. As each edge of $\scr{G}_1$ from $V \setminus Y_{\le 1}$
into $Y_{\le 1}$ occurs with probability at most $p$, $\scr{G}_2$ can be constructed such that $\scr{G}_1 \subseteq \scr{G}_2$,
while maintaining the distributional properties we desire.
Combining this construction with the observation that the alternative forcing process is edge monotonic, we
get that $pt_{pzf}( \scr{G}_1 , Y_{\le 1}) \ge pt_{pzf}( \scr{G}_2, Y_{\le 1})$. As a result, any lower bound
on $pt_{pzf}( \scr{G}_2, Y_{\le 1})$ yields a lower bound on $pt_{pzf}( \scr{G}_1 , Y_{\le 1})$. Moreover, 
the structure of $\scr{G}_2$ is amenable to an application of Lemma \ref{lem:forcing_independence}.

Let us now summarize the main result we shall use throughout the proof of the lower bound in the sparse
regime. We remark that while the result is stated for round $i \ge 0$ of the forcing process in $\Gnp$,
the above coupling of $\scr{G}_1 \subseteq \scr{G}_2$ allows us to assume that $i=0$ in the proof of the
statement. The proof of Corollary \ref{cor:forcing_independence} can be found in Section \ref{sec:auxiliary}.

\begin{corollary} \label{cor:forcing_independence}
Suppose that $Y_{\le i} \subseteq V$ consists of the blue vertices after $i \ge 0$ 
rounds. Let $Y_{i+1}= \bigcup_{u \in Y_{\le i}} Y_{i+1}(u)$ denote the blue vertices forced in round $i+1$,
where $Y_{i+1}(u)$ corresponds to the white vertices forced by $u$ in round $i+1$.
If $S \subseteq V \setminus Y_{\le i}$ and we condition on the
event in which $Y_{i+1}=S$, as well as $Y_{\le i}$ and the edges within it, then $e(Y_{\le i},Y_{i+1})$ is stochastically upper bounded by $|S|+  \Bin( |Y_{\le i}||S|, p)$. Moreover, if $S^* \subseteq V \setminus (Y_{\le i} \cup S)$, then $e(Y_{\le i}, S \cup S^*)$ is stochastically upper bounded by $|S| + \Bin( |Y_{\le i}||S| + |Y_{\le i}||S^*|, p)$ given $Y_{i+1}=S$ (and $Y_{\le i}$ together with the edges within $Y_{\leq i}$).
\end{corollary}

We are now ready to prove our lower bound in the sparse regime. It will be convenient 
to define $\til{d}_U := (1 + \omega^{-1}) d$, where $d:=(n-1)p$, as a means to $\aas$ upper bound the maximum degree $\Delta$
of $\Gnp$. The result will follow by considering $C_2 \to 1$ and $C_1 \to 0$, both tending to the corresponding constants slowly enough as $n \rightarrow \infty$,
such that the asymptotic computations in the below argument continue to hold. For instance, we may take $C_{1}(n) \gg \max\{ 1/\log \log n , \log \log \log (1/p) / \log(1/p)\}$ and $1-C_{2}(n) \gg \log \log(1/p)/ \log(1/p)$, while still ensuring that $C_{1} \to 0$ and $C_{2} \to 1$.

\begin{theorem}\label{thm:lower_bound_sparse}
Suppose that $\log(n)/n \ll p=p(n) \le 1 / \log^2 n$ and $0<C_{1}< C_{2} <1$. Then, if $v \in V$ is fixed, the following
bound holds a.a.s.: 
\[
pt_{pzf}(\Gnp,v) \ge (C_2 - C_1) \log_{4}(1/p).
\]
\end{theorem}

\begin{proof}
As in the proof of Theorem \ref{thm:lower_bound_dense}, we begin with a subset $Y_{0} \subseteq V$
of size $b_{0}:= 1/ {p}^{C_1}$  consisting of the vertices of the graph which
begin initially blue, including the given $v \in V$ (observe $b_{0} < n$ as $p \gg \log n/n$). The forcing process is then started, and for each $i \in \mathbb{N}$ we denote $Y_{i} \subseteq V$
as the vertices of $\Gnp$ which are turned blue in round $i$.
If we fix $i \ge 0$, then $Y_{ \le i}:= \cup_{j=0}^{i}Y_{j}$ consists of the blue vertices
of $\Gnp$ after the first $i$ rounds.

We may define the stopping time $\tau \ge 0$ to be the
first $i \ge 0$ such that $|Y_{ \le i}| > 2(1/p)^{C_2}$ (note $2(1/p)^{C_2}<n$ since $p\gg \log n/n$). Our goal is to show that $\aas$ $\tau \ge (C_2 - C_1)\log_{4}(1/p)$. 
It will be convenient to once again control how the rounds of the forcing process progress. In order to do so,
we first introduce some notation.
For each $j \ge 0$, let $b_{j}:= |Y_{\le j}|$, $\Avg(Y_{\le j}):= \sum_{u \in Y_{\le j}} \deg_{Y_{\le j}}(u) / |Y_{\le j}|$, and 
$\eps = \eps(n):= p^{C_1/3}$. (Note that $\eps=o(1)$, as $p = o(1)$.) Also, recursively define each $\eta_{j}$, where $\eta_{0}:=0$ and
\begin{equation} \label{eqn:recursive_error_definition}
\eta_{j+1}:= \left(\frac{3}{4} + \frac{\eps}{2}\right)\eta_{j}+ \frac{3}{2} \eps + 6p^\frac{1 - C_2}{2}
\end{equation}
for $j \ge 0$. 
We then say that round $j \ge 1$ is \textit{good} provided
the following two properties are satisfied:
\begin{enumerate} \label{eqn:good_average_degree}
    \item[(P1)] $|Y_j| = (3 + \eta_{j-1})(1+ \eps) \, b_{j-1}$,
    \item[(P2)] $\Avg(Y_{ \le j}) \le 2 + \eta_{j}$.
\end{enumerate}

Observe that if the first $t_{F}:=(C_2 - C_1)\log_{4}(1/p)$ many rounds satisfy (P1), then 
\begin{align*}
    |Y_{\le t_{F}}| &= \prod_{j=1}^{t_{F}} [ 1 + (3 + \eta_j)(1 + \eps)](1/p)^{C_1} \\
            & = \prod_{j=1}^{t_F} [4 + 3 \eps + \eta_j + \eta_j \eps] (1/p)^{C_1}  \\
            & \le [4 + 7 \eps + 12 p^{\frac{1 - C_2}{2}}]^{t_{F}}(1/p)^{C_1} \\
            &=  4^{t_{F}}\left[1 + \frac{7}{4} \eps + 3p^{\frac{1 - C_2}{2}}\right]^{t_{F}}(1/p)^{C_1},
\end{align*}
as $\eta_j \le 3 \eps + 12 p^\frac{1 - C_2}{2}$ for each $ j \ge 0$. Yet $\eps = p^{C_1/3}, p^{\frac{1 - C_2}{2}} \rightarrow 0$
as $n \rightarrow \infty$, and so since $t_{F} = O( \log( 1/p))$, 
\begin{align*}
   \left[ 1 + \frac{7}{4}\eps + 3 p^{\frac{1 - C_2}{2}}\right]^{t_{F}} &= \exp \left[t_{F}\left( \frac{7}{4}\eps + 3p^{\frac{1 - C_2}{2}}\right) (1+o(1)) \right]  \\
&= e^{o(1)} = 1 +o(1).
\end{align*}


As a result, 
\[
|Y_{\le t_F}| \leq (1 + o(1)) 4^{t_F}(1/p)^{C_1}=(1+o(1))(1/p)^{C_2}
\]
after substituting $t_{F}= (C_2 - C_1) \log_{4}(1/p)$, and so we have that $|Y_{\leq t_{F}}| \leq 2 (1/p)^{C_2}$.
Thus, under these conditions, $\tau \ge t_{F}$. If we can therefore prove that $\aas$\ the first $(C_2 - C_1)\log_{4}(1/p)$ many rounds are good,
then the proof will be complete.

We begin by considering the number of edges with both endpoints in $Y_0$. Observe
that
\[
    \mb{E} \, \frac{e(Y_0)}{ |Y_0|}  = \frac{\binom{|Y_0|}{2} p}{ |Y_0| } = \Theta( |Y_0|p )  = \Theta( p^{1 - C_1}) = o(1),
\]
as $p \le 1 / \log^2 n$ and $C_{1} <1$ by assumption. Thus, we may use Markov's inequality
to conclude that $\aas$\ $\Avg(Y_0) \le 2$. When this event holds, we say that round $0$
is good. We shall assume that this is the case in what follows.

Let us now fix $i \ge 0$ and condition on $Y_{\le i}$ as well the edges within it. Under the assumption
that the first $i$ rounds are good, we shall lower bound the probability that round $i+1$ is good as well.
In addition to the previous information, let us now condition on the random subset $Y_{i+1} \subseteq V \setminus Y_{\le i}$.
If $|Y_{i+1}| > (3 + \eta_i)(1+ \eps) |Y_{\le i}|$, then the round is \textit{bad} and we stop the analysis (we shall later
show that this occurs with sufficiently small probability). Otherwise, we continue and show that (P2) holds with probability at least
$1 - 2p^{\frac{1 - C_2}{2}}$.

Our goal at this point is to control the average degree of $Y_{\le i+1}$, namely
$\Avg(Y_{\le i+1})$. We first observe that we may decompose
the random variable $e(Y_{\le i+1})$ into three terms where
\[
    e(Y_{\le i+1}) = e(Y_{\le i}) + e(Y_{\le i}, Y_{i+1}) + e(Y_{i+1}).
\]
Observe that since round $i$ was assumed to be good, we know by (P2) that 
\[
2 \, e(Y_{\le i}) \le (2 + \eta_i) |Y_{\le i}|.
\]
Moreover, we have assumed that $|Y_{i+1}| \le (3 + \eta_{i})(1+ \eps) |Y_{\le i}|$, so we can apply Lemma~\ref{lemma coupling} and add white vertices to $Y_{i+1}$ to ensure that 
$|Y_{i+1}| = (3 + \eta_i)(1 + \eps) |Y_{\le i}|$. Thus, we have that $|Y_{\le i+1}| = (1+(3 + \eta_i)(1 + \eps))|Y_{\le i}|$. As a result,
\[
2\, e(Y_{\le i}) \le \frac{(2 + \eta_i) |Y_{\le i+1}|}{1+(3 + \eta_i)(1 + \eps)}.
\]

If we now consider $e(Y_{i+1})$, then observe that the edges within $Y_{i+1}$
are distributed as $\Bin( \binom{|Y_{i+1}|}{2}, p)$. Thus,
\[
 \frac{ 2 \, \mb{E} \, e(Y_{i+1})}{ |Y_{i+1}|} \le  p |Y_{i+1}|.
\]
Since we have assumed that the first $i$ rounds satisfy property (P2) and that $i< t_{F}$, we know that $|Y_{i+1}| \le 2(1/p)^{C_2}$. As
a result, $p |Y_{i+1}|  \le 2p^{1-C_2}$. By Markov's inequality, this implies that 
\[
\frac{ 2 \, e(Y_{i+1})}{ |Y_{\le  i+1}|} \le 2p^{\frac{1 - C_2}{2}}
\]
with probability at least $1 - p^{\frac{1 - C_2}{2}}$, as $C_{2} < 1$.

It remains to control the edges between $Y_{\le i}$ and $Y_{i+1}$. 
By Corollary \ref{cor:forcing_independence}, we know that 
\[
 e(Y_{\le i}, Y_{i+1}) \preceq  |Y_{i+1}| + \Bin(|Y_{\le i}||Y_{i+1}|,p).
\]
Now if $X \in \Bin(|Y_{\le i}||Y_{i+1}|,p)$, then $\mb{E} X \, / |Y_{\le i+1}| \le p |Y_{i+1}|$. As before, we may conclude
that 
\[
\frac{X}{|Y_{\le i+1}|} \le 2p^{\frac{1 - C_2}{2}}
\]
with probability at least $1-p^{\frac{1 - C_2}{2}}$ by Markov's inequality. On the other hand,
$$
|Y_{i+1}|/ |Y_{\le i+1}| = (3 + \eta_i)(1 + \eps)/ (1 + (3 + \eta_i)(1 + \eps))
$$ 
as $|Y_{\le i+1}| = (1 + (3 + \eta_i)(1 + \eps)) |Y_{\le i}|$.
To conclude, we get that
\begin{align*}
    \Avg(Y_{\le i+1}) &= \frac{2[ e(Y_{ \le i}) + e( Y_{\le i}, Y_{i+1}) + e(Y_{\le i+1}) ]}{ |Y_{\le i+1}|} \\
                      &\le \frac{2 + \eta_i}{1 + (3 + \eta_i)(1 + \eps)} + \frac{2((3 + \eta_i)(1 + \eps))}{1 + (3 + \eta_i)(1 + \eps)} + 6p^{\frac{1-C_2}{2}} \\
                      &= \frac{2 + \eta_i}{4 + 3 \eps + \eta_i + \eps \eta_i} + \frac{ 6 + 2( 3\eps + \eta_i + \eps \eta_i)}{4 + 3 \eps + \eta_i + \eps \eta_i} + 6 p^{\frac{1-C_2}{2}} \\
                      &\le 2 + \frac{\eta_{i}}{4} + \frac{3 \eps + \eta_i + \eps \eta_i}{2}  + 6p^{\frac{1 -C_2}{2}}        \\
                      &= 2 + \left(\frac{3}{4} + \frac{\eps}{2}\right) \eta_i + \frac{3}{2}\eps + 6p^{\frac{1-C_2}{2}}            \\
                      &=2 + \eta_{i+1}
\end{align*}
with probability at least $1 - 2 p^{\frac{1-C_2}{2}}$. Thus, round $i+1$ satisfies property (P2) with probability at least $1 - 2p^{\frac{1 - C_2}{2}}$, provided it satisfies (P1)
and the first $i$ rounds are good.

Our next goal will be to lower bound the probability that (P1) is satisfied, under the assumption that the first $i$ rounds are good. We remark that for the purpose of the following analysis, we condition on the outcome of $Y_{\le i}$ and the edges within it,
but \textit{not} the newly forced vertices $Y_{i+1}$. For each $u \in Y_{ \le i}$, we let $Y_{i+1}(u)$ denote the vertices of $V \setminus Y_{\le i}$ which
are forced by $u$ in round $i+1$. Observe that since we know the value of $\deg_{Y_{\le i}}(u)$ for each $u \in Y_{\le i}$, the random variables $\{Y_{i+1}(u)\}_{u \in Y_{\le i}}$ are independent. Moreover, since we may assume that $\Delta( \Gnp) \le \til{d}_U$, it follows that $|Y_{i+1}(u)| \preceq \Bin( \til{d}_{U}, \deg_{Y_{\le i}}[u]/ \til{d}_{L})$
for all $u \in Y_{\le i}$, and so we may conclude that $|Y_{i+1}| \preceq \sum_{u \in Y_{\le i}} \Bin( \til{d}_{U}, \deg_{Y_{\le i}}[u]/ \til{d}_{L})$ (recall that $\til{d}_{U}:=(1+\omega^{-1})d$ and $\til{d}_{L}=(1-\omega^{-1})d$ are upper and lower bounds on the degrees that hold a.a.s.).
Thus, since $\til{d}_{U}/\til{d}_{L} = 1+o(1)$, 
\begin{align*}
    \mb{E} |Y_{i+1}| &\le (1+o(1)) \sum_{u \in Y_{ \le i}} \deg_{Y_{\le i}}[u]    \\
                     &= (1 +o(1))( |Y_{\le i}| + \sum_{u \in Y_{ \le i}} \deg_{Y_{\le i}}(u)) \\
                     &=  (1+o(1)) (1+\Avg(Y_{\le i}))|Y_{\le i}|                 \\
                     & \le (1+o(1)) (3 + \eta_{i}) |Y_{\le i}|,
\end{align*}
where the last line follows from the assumption that round $i$ satisfies (P2). We may now apply a generalized version of Chernoff's bound
\eqref{chern} to $\sum_{u \in Y_{\le i}} \Bin( \til{d}_{U}, \deg_{Y_{\le i}}[u]/ \til{d}_{L})$ with $\eps = p^{C_1/3}$ to conclude that
$|Y_{i+1}| \le (3 + \eta_i)(1+ \eps) |Y_{\le i}|$ 
with probability at least $1 - \exp(- |Y_{\le i}| p^{2 C_1/3}/3) \ge 1 - \exp(-\Theta( (1/p)^{C_1/3}))$. By Lemma \ref{lemma coupling}, we may in fact assume that $|Y_{i+1}| = (3 + \eta_{i})(1+ \eps) |Y_{\le i}|$,
thus showing that (P1) holds for round $i+1$ with probability at least $1 - \exp(-\Theta( (1/p)^{C_1/3}))$, assuming the previous rounds are good.

Combining the above results, the probability that round $i+1$  is bad \textit{and} the previous rounds are good,
is at most
\[
    2 p^{\frac{1-C_2}{2}} + \exp(-\Theta( (1/p)^{C_1/3})) = o( (t_{F})^{-1}).
\]
In order to complete the argument, we need only bound the probability that one of
the first $t_{F}$ many rounds is bad. Observe,
\begin{eqnarray*}
\mb{P}(\cup_{j=1}^{t_{F}}\text{Round $j$ is bad}) & = & \sum_{j=1}^{t_F} \mb{P}(\text{Round $j$ is bad and the earlier rounds are good}) \\
&\le& t_F \cdot o((t_F)^{-1}) = o(1),
\end{eqnarray*}
and so the proof is complete.
\end{proof}


\section{Proofs of Auxiliary Lemmas}\label{sec:auxiliary}

\begin{proof}[Proof of Lemma~\ref{lem:forcing_independence}]
In what follows, we consider the alternative forcing process in which each edge $(u,v) \in \Gnp$ is
forced with probability $\min\{ 1,\deg_{Y_0}[u]/\til{d}_{L}\}$ for $u \in Y_0$ and $v \in V \setminus Y_0$. Moreover,
we expose the edges with both endpoints in $Y_0$ ahead of time, so that $\deg_{Y_0}[u]$ is known for each $u \in Y_0$ (while we condition on these random edges,
we do not include this in the probability notation for simplicity). For convenience, the following
computations are done assuming $\deg_{Y_0}[u]/\til{d}_{L} <1$ for each $u \in Y_0$.

Let us first fix $1 \le j \le k$ and the subset $S_j:=S_{u_j}$ of $V \setminus Y_0$. Observe that we have 
\[
    \mb{P}( Y_1(u_j) = S_j) = \prod_{w_1 \in S_j} \mb{P}(\text{$u_j$ forces $w_1$}) \prod_{w_2 \in V \setminus Y_0 \cup S_j} \mb{P}( \text{$u_j$ does not force $w_2$}).
\]
Using the independence of $\deg_{Y_0}[u_j]$ and the edges from $Y_0$ into $V \setminus Y_0$, we know that
\[
    \mb{P}(\text{$u_j$ forces $w_1$}) = p \left( \frac{\deg_{Y_0}[u_j]}{\til{d}_{L}} \right),
\]
and
\[
    \mb{P}(\text{$u_j$ does not force $w_2$}) = 1 - p \left( \frac{\deg_{Y_0}[u_j]}{\til{d}_{L}} \right),
\]
for each $w_1 \in S_j$ and $w_2 \in V \setminus Y_0 \cup S_j$. 
Thus,
\[
\mb{P}( Y_1(u_j) = S_j ) = p^{|S_j|} \left( \frac{\deg_{Y_0}[u_j]}{\til{d}_{L}} \right)^{|S_j|} \left(1 - p \left( \frac{\deg_{Y_0}[u_j]}{\til{d}_{L}} \right)\right)^{|V| - |Y_0|-|S_j|}.
\]
On the other hand, clearly the random variables $Y_1(u_1), \ldots ,Y_1(u_k)$ are independent,
as the edges in $Y_0$ have already been exposed. It follows that
\[
    \mb{P}( \cap_{j=1}^{k} \{Y_1(u_j) =S_j\}) = \prod_{j=1}^{k} p^{|S_j|} \left( \frac{\deg_{Y_0}[u_j]}{\til{d}_{L}} \right)^{|S_j|} \left(1 - p \left( \frac{\deg_{Y_0}[u_j]}{\til{d}_{L}} \right)\right)^{|V| - |Y_0|-|S_j|}.
\]

Let us now fix some $u \in Y_0$, together with $v \in V \setminus (Y_0 \cup S_{u})$. We shall consider
the probability that the events $\cap_{j=1}^{k} \{Y_1(u_j)=S_j\}$ and $(u,v) \in \Gnp$ both occur. We first observe that
\begin{align*}
    \mb{P}&(\text{$(u,v) \in \Gnp$ and $\cap_{j=1}^{k} Y_1(u_j) = S_j$}) \\
    &= \mb{P}( \text{$(u,v) \in \Gnp$ and $Y(u)=S_u$}) \, \mb{P}( \cap_{u_j \neq u} Y(u_j) = S_j),
\end{align*}
as the events $\{(u,v) \in \Gnp \} \cap \{Y_1(u)=S_u\}$ and $\cap_{u_j \neq u} \{Y_1(u_j) = S_j\}$ are independent.
A similar argument also shows that

\[
 p \left( 1 - \frac{\deg_{Y_0}[u]}{\til{d}_{L}} \right) p^{|S_u|} \left( \frac{\deg_{Y_0}[u]}{\til{d}_{L}} \right)^{|S_u|} \left(1 - p \left( \frac{\deg_{Y_0}[u]}{\til{d}_{L}} \right)\right)^{|V| - |Y_0|-|S_u| + 1}
\]
corresponds to $\mb{P}( \text{$(u,v) \in \Gnp$ and $Y_1(u)=S_u$})$.
After dividing the above probabilities, it follows that
\[
    \mb{P}( (u,v) \in \Gnp | \cap_{j=1}^{k} Y_1(u_j) =S_j) = p \left( 1 - \frac{ \deg_{Y_0}[u]}{\til{d}_{L}} \right)\left(1 - p\frac{ \deg_{Y_0}[u]}{\til{d}_{L}} \right) \le p.
\]
Thus,
\[
\mb{P}( (u,v) \in \Gnp | \cap_{j=1}^{k} Y_1(u_j)=S_j) \le p.
\]
for each $u \in Y_0$ and $v \in V \setminus Y_0 \cup S_u$.
The second property of equation \ref{eqn:forcing_independence} can be verified by generalizing the above computations to collections of vertices outside of $Y_0$.
\end{proof}

\begin{proof}[Proof of Corollary~\ref{cor:forcing_independence}]
We shall assume, without loss of generality, that $i=0$ in what follows.
Let us suppose that $Y_{0}=\{u_1, \ldots ,u_k\}$ for $k \ge 1$, and we fix $S_{1}, \ldots ,S_{k} \subseteq V \setminus Y_0$
together with $S:= \cup_{j=1}^{k}S_j$.
In all of the following computations, we condition on the events $\{Y_1(u_j)=S_j\}_{j=1}^{k}$
and the random variables $\{\deg_{Y_{0}}[u_j]\}_{j=1}^{k}$. For convenience, we omit these assumptions
from the notation in the following computations.

Fix $1 \le j \le k$, and consider the random variable $e(u_j, S)$ which counts the edges from $u_j$
into $S$, conditioned on the above events. Clearly, $e(u_j,S) = e(u_j, S_j) +e(u_j, S \setminus S_j)$,
where the latter random variables are defined analogously. As $S_j$ consists of the vertices outside of
$B$ which are forced by $u_j$, we know that $e(u_j,S_j)=|S_j|$. On the other hand, we may apply
both statements from Lemma \ref{lem:forcing_independence} to conclude that $e(u_j, S \setminus S_j)$
is stochastically upper bounded by $|S_j|+ \Bin(|S \setminus S_j|, p)$.

Observe now that $e(Y_0,S)= \sum_{i=1}^{k}e(u_j,S)$. As the random variables
$\{ e(u_j,S) \}_{j=1}^{k}$ are conditionally independent of the above events by the second
property of Lemma \ref{lem:forcing_independence}, it
follows that $e(Y_0,S)$ is stochastically upper bounded by $|S| + \sum_{i=1}^{k} \Bin( |S| , p)$.
This final random variable is distributed as $|S|+ \Bin(|S||Y_0|,p)$ and so the claim follows.

If we now take $S^* \subseteq V \setminus (B \cup S)$, then Lemma~\ref{lem:forcing_independence} implies that $e(Y_0,S^*)$ is stochastically upper bounded
by $\Bin(|S^*||Y_0|,p)$. The independence of $e(Y_0,S)$ and $e(Y_0,S^*)$ guaranteed 
by equation \ref{eqn:forcing_independence} ensures that $e(Y_0, S \cup S^*)$ is stochastically
upper bounded by $|S| + \Bin( |Y_0||S| + |Y_0||S^*|), p)$.
\end{proof}

\bibliographystyle{plain}
\bibliography{ref}

 


\end{document}